\documentclass[11pt]{amsart}
\usepackage{amssymb}
\usepackage{tikz-cd}
\usepackage{amsrefs}
\usepackage{xfrac}
\usepackage{thmtools}
\usepackage{thm-restate}
\usepackage{tikz}
\usetikzlibrary{decorations.pathmorphing}

\title{Transfinite Schreier families and cardinal invariants}

\author{Kevin Beanland}
\address{Department of Mathematics, Washington and Lee University, Lexington, VA 24450, USA}
\email{beanlandk@wlu.edu}
\urladdr{https://kbeanland.academic.wlu.edu}

\author{Christian Rosendal}
\address{Department of Mathematics, University of Maryland, College Park, MD 20742, USA}
\email{rosendal@umd.edu}
\urladdr{https://sites.google.com/view/christian-rosendal}

\date{}
\newcommand {\N}{\mathbb N}
\newcommand {\Q}{\mathbb Q}

\newcommand{\mgd}[2]{\{  {#1}\;|\; {#2} \} }
\newcommand{\Mgd}[2]{\big\{  {#1}\;\big|\; {#2} \big\} }
\newcommand{\MGd}[2]{\Big\{  {#1}\;\Big|\; {#2} \Big\} }

\newcommand{\om}{\omega}

\newcommand{\tom}{\emptyset}

\newcommand{\saa}{\Rightarrow}
\newcommand{\equi}{\Leftrightarrow}

\newcommand {\go} {\mathfrak}
\newcommand {\ku} {\mathcal}

\newcommand {\e} {\exists}

\newcommand{\Nstar}{\N^\star}
\newcommand{\FINstar}{\mathrm{FIN}^\star}
\newcommand{\Lim}{\mathrm{Lim}}
\newcommand{\dom}{\operatorname{dom}}
\newcommand{\cof}{\operatorname{cof}}
\newcommand{\I}[2]{[#1,#2]}
\newcommand{\bnum}{\go b}
\newcommand{\dnum}{\go d}
\newcommand{\cnum}{\go c}
\newcommand{\Sdot}{\ku S^\bullet}

\newtheorem{thm}{Theorem}
\newtheorem{cor}{Corollary}
\newtheorem{lemme}{Lemma}
\newtheorem{prop} {Proposition}

\begin{document}

\begin{abstract}
We study the dependence of the transfinite Schreier hierarchy on the choice of fundamental sequences for the countable limit ordinals.  With each Schreier family we associate an interval endpoint function.  We prove that the bounding number is equal to $\om_1$ exactly when the fundamental sequences may be chosen so that these endpoint functions form an unbounded family in the eventual domination order.  Equivalently, the hierarchy then satisfies the tail-covering property isolated by Shiliaev, or every infinite compact interval family has infinite intersection with some member of the hierarchy.  We also prove that the dominating number is equal to $\om_1$ exactly when the fundamental sequences may be chosen so that every compact family of finite subsets of $\Nstar=\{2,3,\ldots\}$ is contained in one Schreier family.  The latter equivalence combines Fremlin's cofinality theorem for compact subsets of the rationals with an absorption construction for compact families.
\end{abstract}



\maketitle

%%%%%%%%%%%%%%%%%%%%%%%%%%%%%%

\section{Introduction}

The classical Schreier family was introduced by J.~Schreier \cite{Schreier} to
construct a weakly null sequence no subsequence of which has
norm-convergent Cesàro means \cite{Schreier}.  Much later, D. E. Alspach and S. A. Argyros
introduced the transfinite hierarchy
$(\ku S_\alpha)_{\alpha<\omega_1}$ in their study of the ordinal
complexity of weakly null sequences \cite{AlspachArgyros}. Schreier-type families have since become standard
tools in Banach space theory, typically involved in the construction of spaces with predetermined geometric properties, e.g., \cite{ADKM} or serving as
direct measures of properties of spaces and vector sequences  \cite{AMT,OTW,beanland}.  The construction of these families depends on a choice of, for each countable limit ordinal $\lambda$, an increasing sequence of smaller ordinals increasing to $\lambda$.  Most applications begin by fixing such a system of sequences $\Lambda$  and the corresponding families $\ku S_\alpha=\ku S_\alpha^\Lambda$ once and for all.  The purpose of this note is instead to determine when $\Lambda$ can be chosen so that the resulting hierarchy has prescribed global covering properties.

This issue appears explicitly in recent work of M. Shiliaev \cite{Shiliaev}.  His Condition A asserts the existence of an $\aleph_1$-sized family of nondecreasing functions $\N\to\N$ which is unbounded under eventual domination.  This is precisely the assumption that the {\em bounding number} $\bnum$ equals $\om_1$.   Theorem~\ref{thm:b} extends the results of Shiliaev by giving  equivalent reformulations of $\bnum=\om_1$ in terms of interval endpoint functions and compact interval families.

A stronger universal covering property is governed by the {\em dominating number} $\dnum$.  In Theorem \ref{thm:d}, we prove that $\dnum=\om_1$ exactly when one can choose a single system of fundamental sequences $\Lambda$ such that every compact family of finite subsets of $\Nstar=\{2,3,\ldots\}$ is contained in some $\ku S^\Lambda_\alpha$.  The set-theoretic input is Fremlin's computation
\[
        \cof(\ku K(\Q),\subseteq)=\dnum,
\]
where $\ku K(\Q)$ is the family of compact subsets of $\Q$, ordered by inclusion.  Since the collection $\FINstar=[\Nstar]^{<\om}$ of all finite subsets of $\Nstar$, viewed as a subset of $2^{\Nstar}$, is homeomorphic to $\Q$, it remains to show that one compact family can always be absorbed by extending a prescribed countable partial system of fundamental sequences.  This is detailed in Proposition~\ref{prop:one-compact}.

%%%%%%%%%%%%%%%%%%%%%%%%%%%
%%%%%%%%%%%%%%%%%%%%%%%%%%%
%%%%%%%%%%%%%%%%%%%%%%%%%%%

\section*{Declaration of generative AI use}

The authors conceived of the present project and formulated both theorems in this paper in October 2025. During preparation of the manuscript,
OpenAI's ChatGPT and Codex systems were used (i) to generate draft
proofs of Proposition~\ref{prop:one-compact} and of the implication
{\rm(1)}$\Rightarrow${\rm(2)} in Theorem~\ref{thm:d}; (ii) to
generate an initial draft of the manuscript; and (iii) to assist
in checking mathematical arguments, references, and notation.
The authors independently checked the
relevant literature and every AI-assisted argument, corrected and
revised the generated material as necessary, and take full
responsibility for all statements, proofs, citations, and other content
of the paper. No AI system was used to formulate the statements of the
results.
%%%%%%%%%%%%%%%%%%%%%%%%%%%
%%%%%%%%%%%%%%%%%%%%%%%%%%%
%%%%%%%%%%%%%%%%%%%%%%%%%%%

\section{Schreier families and endpoint functions}
We work over $\Nstar=\{2,3,\ldots\}$ as opposed to $\{1,2,\ldots\}$.  This harmless shift makes the limit clause uniform and allows the non-trivial part of every successor step to involve at least two blocks.
We identify every $A\in\FINstar$ with its characteristic function ${\bf 1}_A\in 2^{\Nstar}$.  Thus compactness of subfamilies of $\FINstar$ always refers to the topology inherited from the compact space $2^{\Nstar}$.

For non-empty $E,F\in\FINstar$, write $E<F$ when $\max E<\min F$.  Also, for $d\in\Nstar$ and non-empty $A\in\FINstar$, write
$d\leqslant A$ when $d\leqslant\min A$.
Consequently,
\[
        d\leqslant A_1<\cdots<A_d
\]
means that $A_1,\ldots,A_d$ are non-empty, $d\leqslant\min A_1$ and $A_1<\cdots<A_d$.

Let $\Lim(\om_1)$ be the set of countable limit ordinals.  A \emph{system of fundamental sequences} is a family
\[
        \Lambda=(\lambda[n])_{\lambda\in\Lim(\om_1),\ n\in\Nstar}
\]
such that, for every $\lambda\in\Lim(\om_1)$,
\[
        \lambda[2]<\lambda[3]<\cdots<\lambda
     =
        \sup_{n\in\Nstar}\lambda[n].
\]
In the Schreier-family literature, these are also called approximating sequences.

Given $\Lambda$, define $(\ku S^\Lambda_\alpha)_{\alpha<\om_1}$ recursively by
\[
        \ku S^\Lambda_0=\{\tom\}\cup\Mgd{\{n\}}{n\in\Nstar},
\]
\[
        \ku S^\Lambda_{\alpha+1}=\ku S^\Lambda_\alpha\cup
        \MGd{\bigcup_{i=1}^d A_i}
        {2\leqslant d\leqslant A_1<\cdots<A_d\;\text{and}\; A_i\in\ku S^\Lambda_\alpha\setminus\{\tom\}},
\]
and, for $\lambda\in\Lim(\om_1)$,
\[
        \ku S^\Lambda_\lambda=\{\tom\}\cup
        \Mgd{A\in\FINstar\setminus\{\tom\}}
        {\e d\in\Nstar\;\big(d\leqslant A\;\;\&\;\;A\in\ku S^\Lambda_{\lambda[d]}\big)}.
\]

For $n\leqslant k$ in $\Nstar$, write $\I{n}{k}=\{n,n+1,\ldots,k\}$.
An \emph{interval family} is a family $\ku K\subseteq\FINstar$ whose non-empty members are intervals.
Recall that $\ku K\subseteq\FINstar$ is \emph{hereditary} if $A\in\ku K$ and $B\subseteq A$ imply $B\in\ku K$, and is \emph{spreading} if
\[
        \{a_1<\cdots<a_r\}\in\ku K,\qquad a_i\leqslant b_i
\]
imply $\{b_1<\cdots<b_r\}\in\ku K$. The following facts are entirely standard and are easily  checked by transfinite induction.

\begin{lemme}\label{lem:finite-successors}
Let $\Lambda$ be a system of fundamental sequences and $\alpha<\omega_1$. Then
\begin{enumerate}
\item $\ku S^\Lambda_\alpha$ is hereditary, spreading, compact and contains all singletons,
\item $\ku S^\Lambda_\alpha\subseteq\ku S^\Lambda_{\alpha+1}$,
\item if $\tom\neq E\in\FINstar$, then $E\in\ku S^\Lambda_{\alpha+(|E|-1)}$.
\end{enumerate}
\end{lemme}

For $\alpha<\om_1$ and $n\in\Nstar$, define
\[
        s^\Lambda_\alpha(n)=
        \max\Mgd{k\in\Nstar}{n\leqslant k\;\&\;\I{n}{k}\in\ku S^\Lambda_\alpha}.
\]
This maximum exists since the defining set is non-empty and finite by compactness of $\ku S^\Lambda_\alpha$.  
Also, for $g,h:\Nstar\to\Nstar$, write $g\leqslant^*h$ when $g(n)\leqslant h(n)$ for all but finitely many $n$.
We shall use the following elementary facts without further comment.

\begin{lemme}\label{lem:elementary}
\begin{enumerate}
\item An interval family $\ku K\subseteq\FINstar$ is compact if and only if either (i) $\ku K$ is finite  or (ii) $\tom\in\ku K$ and, for every $n\in\Nstar$, we have $n\leqslant I$ for all but finitely many $I\in \ku K\setminus \{\tom\}$.
\item Every infinite compact interval family contains successive nonempty intervals $I_1<I_2<\cdots$.
\item If $\ku F\subseteq(\Nstar)^{\Nstar}$ is $\leqslant^*$-unbounded, then $\sup_{f\in\ku F}f(n)=\infty$ for infinitely many $n\in \Nstar$.
\end{enumerate}
\end{lemme}
All three assertions are immediate from the product topology on $2^{\Nstar}$ and the definition of eventual domination $\leqslant^*$.

%%%%%%%%%%%%%%%%%%%%%%%%%%%
%%%%%%%%%%%%%%%%%%%%%%%%%%%
%%%%%%%%%%%%%%%%%%%%%%%%%%%
%%%%%%%%%%%%%%%%%%%%%%%%%%%
%%%%%%%%%%%%%%%%%%%%%%%%%%%
%%%%%%%%%%%%%%%%%%%%%%%%%%%

\section{The bounding number}
  The bounding number $\bnum$ is the least size of a $\leqslant^*$-unbounded family in $(\Nstar)^{\Nstar}$, while the dominating number $\dnum$ is the least size of a $\leqslant^*$-cofinal family.  Thus
\[
        \om_1\leqslant\bnum\leqslant\dnum\leqslant\cnum=2^{\aleph_0}.
\]
Every pattern of equalities and strict inequalities compatible with this chain is relatively consistent with ZFC; see Blass \cite{Blass}.

\begin{lemme}\label{lemme:b}
Fix a system of fundamental sequences $\Lambda$ and consider the following
statements.
\begin{enumerate}
\item For every uncountable $A\subseteq\om_1$, there is $n\in\Nstar$ with
\[
        [\{n,n+1,\ldots\}]^{<\om}
        \subseteq
        \bigcup_{\alpha\in A}\ku S^\Lambda_\alpha.
\]
\item For every infinite compact interval family $\ku K\subseteq\FINstar$, there is $\alpha<\om_1$ for which $\ku K\cap\ku S^\Lambda_\alpha$ is infinite.
\item $(s^\Lambda_\alpha)_{\alpha<\om_1}$ is $\leqslant^*$-unbounded in $(\Nstar)^{\Nstar}$.
\end{enumerate}
Then (1)$\saa$(2)$\saa$(3).
\end{lemme}

\begin{proof}
(1)$\saa$(2):  Suppose $\ku K$ is an infinite compact interval family such that, for all $\alpha<\om_1$, the intersection $\ku K\cap\ku S^\Lambda_\alpha$ is finite. This means that, for every $\alpha<\omega_1$, there is some $n_\alpha\in \Nstar$ so that
$$
n_\alpha\leqslant I\in \ku K\setminus \{\tom\} \;\saa\; I\notin \ku S^\Lambda_\alpha.
$$
There is therefore some uncountable $A\subseteq\om_1$ and $n\in \Nstar$ so that $n_\alpha=n$ for all $\alpha\in A$. Since, for every $m$,  there is some $\max\{n,m\}\leqslant I\in \ku K\setminus \{\tom\}$, we find that
$$
[\{m,m+1,\ldots\}]^{<\om} \not\subseteq \bigcup_{\alpha\in A}\ku S^\Lambda_\alpha
$$
for all $m$.

(2)$\saa$(3): Suppose that $s_\alpha^\Lambda\leqslant^* g$ for all $\alpha<\omega_1$ where $g$ is chosen such that $n\leqslant g(n)$ for all $n\in \Nstar$ and let 
$$
\ku K=\{\tom\}\cup \Mgd{[n,g(n)+1]}{n\in\Nstar}.
$$
Then $\ku K$ is an infinite compact interval family and, if $\alpha<\omega_1$ is given, there is some $m$ so that $s_\alpha^\Lambda(n)\leqslant g(n)$ and hence
$$
[n,g(n)+1]\notin \ku S^\Lambda_\alpha
$$
for all $n\geqslant m$.  It follows that $\ku K\cap \ku S^\Lambda_\alpha$ is finite. 
\end{proof}

\begin{thm}\label{thm:b}
The following assertions are equivalent.
\begin{enumerate}
\item $\bnum=\om_1$.
\item There is a system of fundamental sequences $\Lambda$ such that $(s^\Lambda_\alpha)_{\alpha<\om_1}$ is $\leqslant^*$-unbounded in $(\Nstar)^{\Nstar}$.
\item There is a system of fundamental sequences $\Lambda$ such that, for every infinite compact interval family $\ku K\subseteq\FINstar$, there is $\alpha<\om_1$ for which $\ku K\cap\ku S^\Lambda_\alpha$ is infinite.
\item There is a system of fundamental sequences $\Lambda$ such that, for every uncountable $A\subseteq\om_1$, there is $n\in\Nstar$ with
\[
        [\{n,n+1,\ldots\}]^{<\om}
        \subseteq
        \bigcup_{\alpha\in A}\ku S^\Lambda_\alpha.
\]
\end{enumerate}
\end{thm}

Shiliaev’s Condition A \cite{Shiliaev} is equivalent to \(\mathfrak b=\omega_1\). His Lemma 4.3 and Corollary 3.5 contain the core of the implication from (1) to (4), while Remark 6.1 and the arguments of §6 contain ingredients for the converse obstruction and for the endpoint formulation (2). These equivalences are not stated there. We give a direct proof in the present conventions and add the compact-interval characterization (3).

\begin{proof}
The implication (2)$\saa$(1) is immediate and Lemma \ref{lemme:b} gives the implications (4)$\saa$(3)$\saa$(2). It thus remains to show (1)$\saa$(4).

So suppose that $\bnum=\om_1$ and let $(b_\xi)_{\xi<\om_1}$ be $\leqslant^*$-increasing and unbounded.  Without loss of generality, we may suppose that $n\leqslant b_\lambda(n)$ for all $\lambda<\om_1$ and $n\in \Nstar$.
We first construct $\Lambda$ so that, for all $\lambda\in\Lim(\om_1)$ and $n\in\Nstar$,
$$
n\leqslant [n,b_\lambda(n)]\in \ku S^\Lambda_{\lambda[n]}
$$
and thus $[n,b_\lambda(n)]\in \ku S^\Lambda_{\lambda}$. 

So fix $\lambda\in\Lim(\om_1)$. Pick ordinals $\eta_2<\eta_3<\ldots<\lambda=\sup_n\eta_n$ and define inductively
$$
\lambda[2]=\eta_2+b_\lambda(2), \qquad \lambda[n+1]=\max\big\{\eta_{n+1}, \lambda[n]\big\}+b_\lambda(n+1).
$$
Then $\lambda[2]<\lambda[3]<\ldots<\lambda=\sup_n\lambda[n]$ and, for any $n$, 
$$
\lambda[n]=\alpha+\big(\big|[n,b_\lambda(n)]\big|-1\big)
$$
for some ordinal $\alpha$ and thus 
$$
[n,b_\lambda(n)]\in \ku S^\Lambda_{\lambda[n]}
$$
by Lemma~\ref{lem:finite-successors}. This finishes the construction.

Suppose now that $A\subseteq \omega_1$ is an uncountable set and find some $r\in \omega$ and uncountable $B\subseteq \Lim(\om_1)$ so that $\lambda+r\in A$ for all $\lambda\in B$. Then $(b_\lambda)_{\lambda\in B}$ is $\leqslant^*$-increasing and unbounded and hence $\sup_{\lambda\in B}b_\lambda(n)=\infty$ for some $n\in \Nstar$. Moreover, 
$$
[n, b_\lambda(n)]\in \ku S_{\lambda}^\Lambda\subseteq \ku S_{\lambda+r}^\Lambda\subseteq   \bigcup_{\alpha\in A}\ku S^\Lambda_\alpha
$$
for all $\lambda \in B$, which shows that
$$
[\{n,n+1,\ldots\}]^{<\om} \subseteq   \bigcup_{\alpha\in A}\ku S^\Lambda_\alpha.
$$
\end{proof}

\begin{cor}\label{cor:b-large}
The following assertions are equivalent.
\begin{enumerate}
\item $\bnum>\om_1$.
\item For every system of fundamental sequences $\Lambda$, there are intervals $F_1<F_2<\cdots$ such that
\[
        \mgd{i\geqslant1}{F_i\in\ku S^\Lambda_\alpha}
\]
is finite for every $\alpha<\om_1$.
\end{enumerate}
\end{cor}

\begin{proof}
Since $\om_1\leqslant\bnum$, the first assertion is the negation of Theorem~\ref{thm:b}(1), and hence of Theorem~\ref{thm:b}(3).  From a compact interval family witnessing this negation, extract $F_1<F_2<\cdots$ by Lemma~\ref{lem:elementary}.  Conversely, such a sequence gives the compact interval family
\[
        \{\tom\}\cup\mgd{F_i}{i\geqslant1}.
\]
\end{proof}

%%%%%%%%%%%%%%%%%%%%%%%%%%%
%%%%%%%%%%%%%%%%%%%%%%%%%%%
%%%%%%%%%%%%%%%%%%%%%%%%%%%
%%%%%%%%%%%%%%%%%%%%%%%%%%%
%%%%%%%%%%%%%%%%%%%%%%%%%%%
%%%%%%%%%%%%%%%%%%%%%%%%%%%

\section{Absorbing compact families}

The interval case requires no set-theoretic hypothesis and can be absorbed already at the first limit level.

\begin{prop}\label{prop:interval-omega}
If $\ku K\subseteq\FINstar$ is a compact interval family, then there is a system of fundamental sequences $\Lambda$ such that
\[
        \ku K\subseteq\ku S^\Lambda_\om.
\]
Only the fundamental sequence for $\om$ needs to be prescribed.
\end{prop}

\begin{proof}
For $n\in\Nstar$, set
\[
        m_{\ku K}(n)=\max\big(\{n\}\cup\mgd{k}{\I{n}{k}\in {\ku K}}\big).
\]
This is finite by Lemma~\ref{lem:elementary}.  Choose a strictly increasing function $g:\Nstar\to\om$ such that
\[
        g(n)\geqslant m_{\ku K}(n)-n
        \qquad(n\in\Nstar),
\]
and prescribe $\om[n]=g(n)$.  If $\I{n}{k}\in {\ku K}$, then Lemma~\ref{lem:finite-successors} and finite successor monotonicity give
\[
        \I{n}{k}\in\ku S^\Lambda_{k-n}
        \subseteq\ku S^\Lambda_{g(n)}
        =\ku S^\Lambda_{\om[n]}.
\]
Since $n\leqslant\I{n}{k}$, the limit definition yields $\I{n}{k}\in\ku S^\Lambda_\om$.
\end{proof}

A \emph{partial system of fundamental sequences} is a function $p$ whose domain is a set of countable limit ordinals and which assigns to every $\lambda\in\dom(p)$ a strictly increasing sequence $(\lambda[n])_{n\in\Nstar}$ cofinal in $\lambda$.  A full system $\Lambda$ extends $p$, written $\Lambda\supseteq p$, if it agrees with all these assignments.  For $\ku A\subseteq\FINstar$, write
\[
        p\Vdash\ku A\subseteq\Sdot_\alpha
\]
when $\ku A\subseteq\ku S^\Lambda_\alpha$ for every full system $\Lambda\supseteq p$.  This notation merely records persistence under full extensions; no forcing notion is involved.

\begin{prop}\label{prop:one-compact}
Let $p_0$ be a countable partial system of fundamental sequences and let ${\ku K}\subseteq\FINstar$ be compact.  Then one can find a countable partial system $p\supseteq p_0$ and an ordinal $\alpha<\om_1$ such that
\[
        p\Vdash {\ku K}\subseteq\Sdot_\alpha.
\]
\end{prop}

\begin{proof}
The assertion is trivial when ${\ku K}=\tom$, so suppose that ${\ku K}\neq\tom$.  Replacing ${\ku K}$ by its hereditary closure, we may suppose that ${\ku K}$ is hereditary.  The hereditary closure is compact: it is the projection onto the first coordinate of the closed subset
\[
        \Mgd{(B,A)\in2^{\Nstar}\times {\ku K}}{B\subseteq A}
\]
of the compact space $2^{\Nstar}\times {\ku K}$.
For $s\in {\ku K}$, let
\[
        {\ku K}_s=\{\tom\}\cup
        \Mgd{t\in\FINstar\setminus\{\tom\}}
        {(s=\tom\;\text{or}\;s<t)\;\&\;s\cup t\in {\ku K}}.
\]
Order ${\ku K}$ by end-extension: $s$ precedes $t$ when $s\subsetneq t$ and every element of $t\setminus s$ is larger than every element of $s$.  Thus the immediate successors of $s$ are the sets $s\cup\{m\}\in {\ku K}$ with $m>\max s$, where $\max\tom=1$.  This tree is well founded.  Indeed, the union of an infinite branch would be the limit in $2^{\Nstar}$ of its finite initial segments and would therefore belong to the compact family ${\ku K}$, which is impossible.

We prove, by induction on the rank of $s$, that every countable partial system $q$ has a countable extension $q'$ and there is an ordinal $\alpha<\om_1$ such that
\[
        q'\Vdash {\ku K}_s\subseteq\Sdot_\alpha.
\]
If ${\ku K}_s=\{\tom\}$, take $q'=q$ and $\alpha=0$.

Suppose now that ${\ku K}_s\neq\{\tom\}$ and put
\[
        M_s=\mgd{m\in\Nstar}{m>\max s\;\&\;s\cup\{m\}\in {\ku K}}.
\]
The set $M_s$ is non-empty: if $\tom\neq t\in {\ku K}_s$ and $m=\min t$, then heredity gives $s\cup\{m\}\in {\ku K}$.  Enumerate it increasingly as $(m_j)_{j<\ell}$, where $0<\ell\leqslant\om$.  Starting with $q^{(0)}=q$, apply the induction hypothesis successively to the nodes $s\cup\{m_j\}$.  This gives an increasing sequence of countable partial systems and ordinals $\beta_{m_j}<\om_1$ such that, for each $j<\ell$, 
\[
        q^{(j+1)}\Vdash
        {\ku K}_{s\cup\{m_j\}}\subseteq\Sdot_{\beta_{m_j}}.
\]
Set
\[
        q^*=q\cup\bigcup_{j<\ell}q^{(j+1)}.
\]
The partial system $q^*$ is countable.  Moreover, every full extension of $q^*$ extends each $q^{(j+1)}$, so 
\[
        q^*\Vdash
        {\ku K}_{s\cup\{m\}}\subseteq\Sdot_{\beta_{m}}
\]
for all 
$m\in M_s$. 

Choose a countable ordinal $\theta$ such that
\[
        \theta>\sup\bigl(\dom(q^*)\cup\{\beta_m+1:m\in M_s\}\bigr),
\]
and, for $m\in\Nstar$, put
\[
        \gamma_m=\theta+\om\cdot(m-1),
        \qquad
        \lambda_s=\theta+\om^2.
\]
Since $m\geqslant2$, each $\gamma_m$ is a countable limit ordinal.  The ordinals $\gamma_m$ and $\lambda_s$ lie strictly above every member of $\dom(q^*)$, and $(\gamma_m)_{m\in\Nstar}$ is strictly increasing and cofinal in $\lambda_s$.  For every $m\in M_s$, assign to $\gamma_m$ a fundamental sequence whose first term is
\[
        \gamma_m[2]=\beta_m+1.
\]
These assignments are compatible, since the ordinals $\gamma_m$ are distinct and lie outside $\dom(q^*)$.  Every full system $\Lambda$ extending the enlarged partial system then satisfies
\[
        {\ku K}_{s\cup\{m\}}\subseteq\ku S^\Lambda_{\gamma_m}
        \qquad(m\in M_s).
\]
Indeed, such a system satisfies ${\ku K}_{s\cup\{m\}}\subseteq\ku S^\Lambda_{\beta_m}\subseteq\ku S^\Lambda_{\beta_m+1}$, and every non-empty member of the latter family belongs to $\ku S^\Lambda_{\gamma_m}$ with limit witness $d=2$, since its minimum is at least $2$.  Finally prescribe
\[
        \lambda_s[m]=\gamma_m
        \qquad(m\in\Nstar),
\]
and denote the resulting countable partial system by $q'$.

Let $\Lambda\supseteq q'$ be full and fix $\tom\neq A\in {\ku K}_s$.  Put $m=\min A$ and $B=A\setminus\{m\}$.  Since $s\cup A\in {\ku K}$ and ${\ku K}$ is hereditary, $s\cup\{m\}\in {\ku K}$.  Moreover, $m>\max s$, with the convention $\max\tom=1$.  Thus $m\in M_s$.  If $B=\tom$, then $A=\{m\}\in\ku S^\Lambda_{\lambda_s+1}$.  If $B\neq\tom$, then
\[
        s\cup\{m\}<B
        \quad\text{and}\quad
        (s\cup\{m\})\cup B=s\cup A\in {\ku K},
\]
so $B\in {\ku K}_{s\cup\{m\}}$.  Consequently,
\[
        B\in\ku S^\Lambda_{\gamma_m}
        =\ku S^\Lambda_{\lambda_s[m]}.
\]
As $m\leqslant B$, the limit definition gives $B\in\ku S^\Lambda_{\lambda_s}$.  Also $\{m\}\in\ku S^\Lambda_{\lambda_s}$ and $2\leqslant\{m\}<B$, whence
\[
        A=\{m\}\cup B\in\ku S^\Lambda_{\lambda_s+1}.
\]
Thus $q'\Vdash {\ku K}_s\subseteq\Sdot_{\lambda_s+1}$, completing the rank induction.  Applying the assertion to $s=\tom$ and $q=p_0$ proves the proposition.
\end{proof}

%%%%%%%%%%%%%%%%%%%%%%%%%%%
%%%%%%%%%%%%%%%%%%%%%%%%%%%
%%%%%%%%%%%%%%%%%%%%%%%%%%%
%%%%%%%%%%%%%%%%%%%%%%%%%%%
%%%%%%%%%%%%%%%%%%%%%%%%%%%
%%%%%%%%%%%%%%%%%%%%%%%%%%%

\section{The dominating number}

For a topological space $X$, let $\ku K(X)$ denote the family of compact subsets of $X$, ordered by inclusion.  The space $\FINstar$ is countable, metrizable and has no isolated points: every basic neighbourhood of a finite set fixes only finitely many coordinates and therefore contains a strictly larger finite set.  Hence $\FINstar$ is homeomorphic to $\Q$ by Sierpi\'nski's theorem \cite{Sierpinski}; see also the modern proof in \cite{Dashiell}.  Fremlin proved \cite{Fremlin} that
\[
        \cof(\ku K(\Q),\subseteq)=\dnum;
\]
this computation is also recalled in \cite{GartsideMamatelashvili}.  Consequently,
\[
        \cof(\ku K(\FINstar),\subseteq)=\dnum.
\]

\begin{prop}\label{prop:fixed-lambda}
For a fixed system of fundamental sequences $\Lambda$, the following are equivalent.
\begin{enumerate}
\item $(s^\Lambda_\alpha)_{\alpha<\om_1}$ is $\leqslant^*$-cofinal in $(\Nstar)^{\Nstar}$.
\item For every compact interval family $\ku K\subseteq\FINstar$, there is $\alpha<\om_1$ such that $\ku K\setminus\ku S^\Lambda_\alpha$ is finite.
\end{enumerate}
\end{prop}

\begin{proof}
Assume (1) and let $\ku K$ be a compact interval family.  Define
\[
        g_{\ku K}(n)=\max\big(\{n\}\cup\mgd{k}{\I{n}{k}\in \ku K}\big).
\]
By Lemma~\ref{lem:elementary}, $g_{\ku K}(n)$ is finite.  Choose $\beta<\om_1$ with $g_{\ku K}\leqslant^*s^\Lambda_\beta$.  Every interval in $\ku K$ with sufficiently large left endpoint $n$ is then a subset of $\I{n}{s^\Lambda_\beta(n)}$ and hence belongs to $\ku S^\Lambda_\beta$ by heredity.  Compactness leaves only finitely many intervals with bounded left endpoint, so $\ku K\setminus\ku S^\Lambda_\beta$ is finite.

Conversely, suppose (2), let $u:\Nstar\to\Nstar$, and put $g(n)=\max\{u(n),n\}$.  The family
\[
        \ku K=\{\tom\}\cup\mgd{\I{n}{g(n)}}{n\in\Nstar}
\]
is compact by Lemma~\ref{lem:elementary}.  Choose $\alpha<\om_1$ so that $\ku K\setminus\ku S^\Lambda_\alpha$ is finite.  Then, for all but finitely many $n$,
\[
        \I{n}{g(n)}\in\ku S^\Lambda_\alpha,
\]
and therefore
\[
        u(n)\leqslant g(n)\leqslant s^\Lambda_\alpha(n).
\]
Thus $u\leqslant^*s^\Lambda_\alpha$, proving cofinality.
\end{proof}

\begin{thm}\label{thm:d}
The following assertions are equivalent.
\begin{enumerate}
\item $\dnum=\om_1$.
\item There is a system of fundamental sequences $\Lambda$ such that, for every compact $\ku K\subseteq\FINstar$, there is $\alpha<\om_1$ with
\[
        \ku K\subseteq\ku S^\Lambda_\alpha.
\]
\item There is a system of fundamental sequences $\Lambda$ such that $(s^\Lambda_\alpha)_{\alpha<\om_1}$ is $\leqslant^*$-cofinal in $(\Nstar)^{\Nstar}$.
\item There is a system of fundamental sequences $\Lambda$ such that, for every compact interval family $\ku K\subseteq\FINstar$, there is $\alpha<\om_1$ for which $\ku K\setminus\ku S^\Lambda_\alpha$ is finite.
\end{enumerate}
\end{thm}

\begin{proof}
Assume $\dnum=\om_1$.  By the preceding cofinality identity, choose a cofinal family $(\ku K_\xi)_{\xi<\om_1}$ in $\ku K(\FINstar)$.  We construct a continuous increasing sequence $(p_\xi)_{\xi<\om_1}$ of countable partial systems.  Let $p_0=\tom$.  At a limit stage $\delta<\om_1$, put
\[
        p_\delta=\bigcup_{\xi<\delta}p_\xi,
\]
which is countable since $\delta$ is countable.  Having constructed $p_\xi$, apply Proposition~\ref{prop:one-compact} to obtain a countable extension $p_{\xi+1}\supseteq p_\xi$ and $\alpha_\xi<\om_1$ such that
\[
        p_{\xi+1}\Vdash \ku K_\xi\subseteq\Sdot_{\alpha_\xi}.
\]
The union
\[
        p^*=\bigcup_{\xi<\om_1}p_\xi
\]
is a partial system of fundamental sequences, since the recursion is increasing and therefore assigns at most one sequence to each limit ordinal.  Extend it arbitrarily to a full system $\Lambda$.  If $\ku K\subseteq\FINstar$ is compact, choose $\xi<\om_1$ with $\ku K\subseteq \ku K_\xi$.  Then
\[
        \ku K\subseteq \ku K_\xi\subseteq\ku S^\Lambda_{\alpha_\xi}.
\]
Thus (1) implies (2).

The implication (2)$\saa$(4) is immediate.  Proposition~\ref{prop:fixed-lambda} gives (3)$\equi$(4) for every fixed $\Lambda$.  Finally, (3) implies $\dnum=\om_1$, since it provides a $\leqslant^*$-cofinal family of size $\om_1$ and always $\om_1\leqslant\dnum$.
\end{proof}

\begin{cor}\label{cor:d-large}
The following assertions are equivalent.
\begin{enumerate}
\item $\dnum>\om_1$.
\item For every system of fundamental sequences $\Lambda$, there is a compact interval family $\ku K\subseteq\FINstar$ such that
\[
        \ku K\setminus\ku S^\Lambda_\alpha\text{ is infinite}
\]
for every $\alpha<\omega_1$.
\end{enumerate}
\end{cor}

\begin{proof}
Since $\om_1\leqslant\dnum$, this is precisely the negation of Theorem~\ref{thm:d}(4).
\end{proof}


\begin{bibdiv}
\begin{biblist}




\bib{AlspachArgyros}{article}{
   author={Alspach, Dale E.},
   author={Argyros, Spiros A.},
   title={Complexity of weakly null sequences},
   journal={Dissertationes Math. (Rozprawy Mat.)},
   volume={321},
   date={1992},
   pages={1--44},
}


\bib{ADKM}{article}{
   author={Argyros, S. A.},
   author={Deliyanni, I.},
   author={Kutzarova, D. N.},
   author={Manoussakis, A.},
   title={Modified mixed Tsirelson spaces},
   journal={J. Funct. Anal.},
   volume={159},
   date={1998},
   number={1},
   pages={43--109},
}



\bib{AMT}{article}{
   author={Argyros, S. A.},
   author={Mercourakis, S.},
   author={Tsarpalias, A.},
   title={Convex unconditionality and summability of weakly null sequences},
   journal={Israel J. Math.},
   volume={107},
   date={1998},
   pages={157--193},
}





\bib{beanland}{article}{
   author={Beanland, Kevin},
   author={Chu, H\`ung Vi\d{\^e}t},
   title={Schreier families and $\mathcal {F}$-(almost) greedy bases},
   journal={Canad. J. Math.},
   volume={76},
   date={2024},
   number={4},
   pages={1379--1399},
}



\bib{Blass}{misc}{
   author={Blass, Andreas},
   title={Combinatorial cardinal characteristics of the continuum},
   note={in \emph{Handbook of Set Theory}, M. Foreman and A. Kanamori (eds.), Springer, Dordrecht, 2010, pp. 395--489},
}

\bib{Dashiell}{article}{
   author={Dashiell, Frederick K., Jr.},
   title={Countable metric spaces without isolated points},
   journal={Amer. Math. Monthly},
   volume={128},
   date={2021},
   number={3},
   pages={265--267},
  }

\bib{Fremlin}{article}{
   author={Fremlin, D. H.},
   title={Families of compact sets and Tukey's ordering},
   journal={Atti Sem. Mat. Fis. Univ. Modena},
   volume={39},
   date={1991},
   number={1},
   pages={29--50},
}

\bib{GartsideMamatelashvili}{article}{
   author={Gartside, Paul},
   author={Mamatelashvili, Ana},
   title={Tukey order, calibres and the rationals},
   journal={Ann. Pure Appl. Logic},
   volume={172},
   date={2021},
   number={1},
   pages={102873},
 }



\bib{OTW}{article}{
   author={Odell, Edward},
   author={Tomczak-Jaegermann, Nicole},
   author={Wagner, Roy},
   title={Proximity to $\ell_1$ and distortion in asymptotic $\ell_1$ spaces},
   journal={J. Funct. Anal.},
   volume={150},
   date={1997},
   number={1},
   pages={101--145},
}


\bib{Schreier}{article}{
   author={Schreier, J.},
   title={Ein {G}egenbeispiel zur {T}heorie der schwachen {K}onvergenz},
   journal={Studia Math.},
   volume={2},
   date={1930},
   number={1},
   pages={58--62},
 
}


\bib{Shiliaev}{misc}{
   author={Shiliaev, Mark},
   title={On unconditionality and higher-order Schreier unconditionality},
   note={arXiv:2510.02783v1 [math.FA]},
   date={2025},
}

\bib{Sierpinski}{article}{
   author={Sierpi\'nski, Wac\l aw},
   title={Sur une propri\'et\'e topologique des ensembles d\'enombrables denses en soi},
   journal={Fund. Math.},
   volume={1},
   date={1920},
   pages={11--16},
 }

\end{biblist}
\end{bibdiv}
\end{document}